\documentclass[runningheads, envcountsame, a4paper]{llncs} 
\usepackage{makeidx}  

\usepackage[pdftex]{graphicx} 
\usepackage{epstopdf} 
\usepackage{amsmath,amsfonts,mathrsfs}
\usepackage{tikz}
\usepackage{amssymb,amscd}
\usepackage[hyphens]{url} 
\usepackage{microtype}
\usetikzlibrary{arrows, automata,chains,fit,shapes}
\tikzstyle{decision} = [diamond, draw, fill=blue!20,
text width=4.5em, text badly centered, node distance=2.5cm, inner sep=0pt]
\tikzstyle{block} = [rectangle, draw, fill=blue!20,
text width=5em, text centered, rounded corners, minimum height=4em]
\tikzstyle{line} = [draw, thick, color=black!50, -latex']
\tikzstyle{cloud} = [draw, ellipse,fill=red!20, node distance=2.5cm,
minimum height=2em]

\newcommand{\N}{\mathbb N}

\begin{document}
\frontmatter          
\pagestyle{headings}  
\addtocmark{Hamiltonian Mechanics} 

\title{Permutations sorted by a finite \\and an infinite stack in series}

 \toctitle{Permutations sorted by a finite and an infinite stack in series}
  \tocauthor{Murray~Elder \and Yoong~Kuan~Goh}

\author{Murray Elder\inst{1}   
\and
Yoong Kuan Goh\inst{2}}
\authorrunning{Murray Elder and Yoong Kuan Goh} 
%
\tocauthor{Murray Elder and Yoong Kuan Goh}
\institute{University of Technology Sydney, Ultimo NSW 2007, Australia
\email{murray.elder@uts.ed.au}
\and
 University of Technology Sydney, Ultimo NSW 2007, Australia  \email{gohyoongkuan@gmail.com}}

\maketitle

\begin{abstract}
 We prove that the set of permutations sorted by a stack of depth $t \geq 3$ and an infinite stack in series has infinite  basis, by constructing an infinite antichain. This answers an open question on identifying the point at which, in a sorting process with two stacks in series, the basis changes from finite to infinite. 
\end{abstract}

\section{Introduction} 
A permutation is an arrangement of an ordered set of elements.  
Two permutations with same relative ordering are said to be \textit{order isomorphic}, for example,  $132$ and $275$ are order isomorphic  as they have relative ordering $ijk$
where $i<k<j$. 
A subpermutation of a permutation $p_{1} \dots p_{n}$ is a word $p_{i_{1}} \dots p_{i_{s}}$ with $1\leq i_{1}< \dots <i_{s}\leq n$.
A permutation $p$ {\em contains} $q$ if it has  a subpermutation  that is order isomorphic to $q$. For example, $512634$ contains $231$ since the  subpermutation $563$  is order isomorphic to $231$. A permutation that  does not contain $q$ is said to  {\em avoid} $q$. 
Let $S_n$ denote the set of permutations of $\{1,\dots, n\}$ and let $S^\infty=\bigcup_{n\in \N_+} 
S_n$.
 The set of all permutations in $S^\infty$ which avoid every permutation in  $\mathscr B\subseteq S^\infty$ is denoted $Av(\mathscr B)$. 
A set  of permutations is a {\em pattern avoidance class} if it equals $Av(\mathscr B)$ for some $\mathscr B\subseteq S^\infty$.
A set $\mathscr B=\{q_{1},q_{2}, \dots\}\subseteq S^\infty$ is an {\em antichain} if no $q_{i}$ contains $q_{j}$  for any $i\neq j$. 
An antichain $\mathscr B$ is a {\em basis} for a 
 pattern avoidance class $\mathscr C$ if $\mathscr C=Av(\mathscr B)$.

Sorting mechanisms are natural sources of pattern avoidance classes, since (in general) if a permutation cannot be sorted then neither can any permutation containing it.
Knuth characterised the set of permutations that can be sorted by a single pass through an infinite stack as the set of permutations that avoid 231 \cite{Knuth}. Since then many variants of the problem have been studied, 
 for example  \cite{AlbertStacksDeques,AB,AtLTMR1453845,AtkinsonMR1932896,Bona2003,ClaessonMR2601799,Estacks,ELRstacks,AndrewMR3573219,pushall,MR3627423,SmithMR3206158,RebeccaEnumerationPopStacks,TarjanMR0298803,West1993}. 
The set of permutations sortable by a stack of depth 2 and an infinite stack in series has a basis of 20 permutations \cite{Estacks}, while for two infinite stacks  in series  there is no finite basis \cite{MurphyThesis}. For systems of a finite stack of depth $3$ or more and infinite stack in series, it was not known whether the basis was finite or infinite.

Here we show that for depth $3$ or more the basis  is  infinite.
We identify an infinite antichain  belonging to the basis
of the set of permutations sortable by a stack of depth  $3$ and an infinite stack in series. A simple lemma then implies the result for depth $4$ or more.
A computer search by the authors (\cite{GohThesis}) yielded $8194$ basis permutations of lengths up to $13$ 
(see Table~\ref{table:basis3};  basis permutations are listed at  \url{https://github.com/gohyoongkuan/stackSorting-3}). The antichain used to prove our theorem was found by examining this data and looking for patterns that could be arbitrarily extended.

\begin{table}\centering
\caption{Number of basis elements for $S(3,\infty)$ of length up to $13$}\label{table:basis3}
\begin{tabular}{llllllllll}
\hline\noalign{\smallskip}
Permutation length \:  & Number of sortable permutations \: & Number of basis elements\\  \noalign{\smallskip}
\hline
\noalign{\smallskip}
		5  & 120 & 0 \\ 
	
		6  & 711& 9 \\
	
		7 & 4700 & 83 \\
	
		8 &  33039& 169 \\
	
		9 & 239800 &  345 \\

		10 &1769019  &  638 \\

		11 & 13160748& 1069 \\
	
		12 &98371244 & 1980 \\
	
		13 & 737463276 &  3901 \\
		
		 \hline
\end{tabular}
\end{table}

\section{Preliminaries}
The notation $\N$ denotes the non-negative integers $\{0,1,2,\dots\}$ and $\N_+$ the positive integers $\{1,2,\dots\}$.

Let $M_t$ denote the machine consisting of a  stack, $R$, of depth $t\in\N_+$ and infinite stack, $L$, in series as in Fig.~\ref{fig:M3}.  A {\em sorting process} is the process of moving entries of a permutation from right to left from the input to stack $R$, then to stack $L$, then to the output, in some order. Each item must pass through both stacks, and at all times stack $R$ may contain no more than $t$ items (so if at some point stack $R$ holds $t$ items, the next input item cannot enter  until an item is moved from $R$ to $L$).

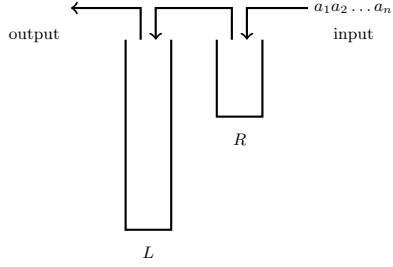
\begin{figure}[ht]
\begin{center}
\begin{tikzpicture}[thick,scale=0.6, every node/.style={scale=0.7}]

\draw (5,1.8) node {input};
\draw (-2,1.8) node {output};

\draw[thick] (0,1.7) -- (0,-2.5) -- (1,-2.5) -- (1,1.7);
\draw (2.5,-.5) node {$R$};
\draw[thick] (2,1.7) -- (2,0) -- (3,0) -- (3,1.7);
\draw (0.5,-3) node {$L$};
\draw[ ->] (4,2.4) -- node[above] {} (2.66,2.4) -- (2.66,1.7);
\draw (5,2.4) node {$a_1a_2\dots a_n$};
\draw[bend left, ->] (2.33,1.7) -- (2.33,2.4) -- node[above] {} (0.66,2.4) -- (0.66,1.7);
\draw[ <-] (-1.2,2.4) -- node[above] {} (0.33,2.4) -- (0.33,1.7);

\end{tikzpicture}
\caption{A stack $R$ of depth $t$ and an infinite stack $L$ in series}
\label{fig:M3}
\end{center}
\end{figure}

A permutation $\alpha = a_1a_2\dots a_n$ is in  $S(t,\infty)$ if it can be sorted to $123\dots n$ using $M_t$. For example, 
$243651\in S(t,\infty)$ for $t\geq 3$ since it can be sorted using the following process: place $2,4$ into stack $R$,  move $4,3,2$ across to stack $L$, place $6,5,1$ into stack $R$, then output $1,2,3,4,5,6$. 
Note $243651\not\in S(2,\infty)$ by   \cite{Estacks}.

The following   lemmas will be used to prove our main result.

\begin{lemma}
	\label{lem:Lrelated} 
	Let $\alpha = a_1a_2\dots a_n\in S(t,\infty)$ for $t\in\N_+$. If $i<j$ and $a_i<a_j$ then in any sorting process that sorts $\alpha$, if both $a_i$ and $a_j$ appear together in stack $L$ then $a_i$ must be above $a_j$.
\end{lemma}
\begin{proof} If $a_j$ is above $a_i$ in stack $L$  then the permutation will fail to be sorted.
\qed\end{proof}

\begin{lemma}
	\label{lem:sbm}
	Let $\alpha = a_1a_2\dots a_n\in S(t,\infty)$ for $t\geq 3$ and suppose $1\leq i<j<k\leq n$ with  $a_i a_j a_k$ order-isomorphic to $132$. Then  in any  sorting process that sorts $\alpha$, $a_i,a_j,a_k$ do not appear together in stack $R$.
\end{lemma}
\begin{proof}
	If $a_i,a_j,a_k$ appear together in stack $R$, we must move $a_k$ then $a_j$ onto stack $L$ before we can move $a_i$, but this means $a_j,a_k$ violate  Lemma~\ref{lem:Lrelated}. \qed
\end{proof}

\begin{lemma}
	\label{lem:67R} 
	Let $\alpha = a_1a_2\dots a_n\in S(t,\infty)$ for $t\geq 3$ and $1\leq i_1<i_2<\dots <i_6\leq n$ with $a_{i_1}a_{i_2}\dots a_{i_6}$ order isomorphic to $243651$. Then in any  sorting process that sorts $\alpha$,  at some step of the process $a_{i_4}$ and $a_{i_5}$ appear together in stack  $R$.
\end{lemma}
\begin{proof}
For simplicity let us write  $a_{i_1}=2, a_{i_2}=4, a_{i_3}=3, a_{i_4}=6, a_{i_5}=5, a_{i_6}=1$. Before $6$ is input, $2,3,4$ are in the two stacks in one of the following configurations:
\begin{enumerate}
	\item $2,4,3$ are all in stack $R$. In this case we  violate Lemma \ref{lem:sbm}.
	\item two items are in stack $R$ and one is in stack $L$.  In this case by Lemma~\ref{lem:Lrelated} we cannot move 
	$6$ to stack $L$, so $6$ must placed and kept in stack $R$. If $t=3$ stack $R$ is now full, so $5$ cannot move into the system, and if $t\geq 4$, when $5$ is input we violate Lemma \ref{lem:sbm}.
	\item one item, say $a$,  is  in stack $R$ and two items are in stack $L$. In this case we cannot move $6,5$ into stack $L$ by Lemma \ref{lem:Lrelated} so they remain in stack $R$ on top of $a$,  violating Lemma \ref{lem:sbm}.
	
	\item stack R is empty.
	In this case, $2, 3,4$ must be placed in stack $L$ in order,  else we violate Lemma~\ref{lem:Lrelated}. 	We cannot place $6,5$ into stack $L$ until it is empty, so they must both stay in  stack $R$ until 4 is output.  
\end{enumerate}
In particular, the last case is the only possibility and in this case $a_{i_4}, a_{i_5}$ appear in stack $R$ together.
\qed\end{proof}

\begin{lemma}
	\label{GZigZag32514}
	Let $\alpha = a_1a_2\dots a_n\in S(t,\infty)$ for $t\geq 3$ and suppose $1\leq i_1<i_2<\dots<i_5\leq n$ with  $a_{i_1}a_{i_2}\dots a_{i_5}$ order-isomorphic to $32514$.  Then, in any sorting process that sorts $\alpha$, if   $a_{i_1},a_{i_2}$ appear together in stack $R$, then at some step in the process $a_{i_3},a_{i_4}$ appear together in stack $L$.
\end{lemma}
\begin{proof}
For simplicity let us write  $a_{i_1}=3$, $a_{i_2}=2$, $a_{i_3}=5$, $a_{i_4}=1$, $a_{i_5}=4$. 
Figure~\ref{fig:lemma4} indicates the possible ways to sort these entries, and in the case that $2,3$ appear together in stack $R$ we see that $4,5$ must appear in stack $L$ together at some later point.
\qed\end{proof}
\begin{figure}[ht]
	\begin{center}
		
		\begin{tikzpicture}[scale=2, node distance = 2cm, auto]
		\begin{node} (root)
		{
			\begin{tikzpicture}[thick,scale=0.4, every node/.style={scale=0.6}]
			\draw[thick] (0,1.7) -- (0,-2.5) -- (1,-2.5) -- (1,1.7);
			\draw (2.5,-.5) node {$R$};
			\draw[thick] (2,1.7) -- (2,0) -- (3,0) -- (3,1.7);
			\draw (0.5,-3) node {$L$};
			\draw[ ->] (4,2.4) -- node[above] {} (2.66,2.4) -- (2.66,1.7);
			\draw (5.0,2.4) node {$2514$};
			\draw[bend left, ->] (2.33,1.7) -- (2.33,2.4) -- node[above] {} (0.66,2.4) -- (0.66,1.7);
			\draw[ <-] (-1.2,2.4) -- node[above] {} (0.33,2.4) -- (0.33,1.7);	
			\draw (2.5,0.3) node {$3$};
			\end{tikzpicture}
		};
	\end{node}
	
	\node[right of=root, node distance=4.0cm] (root_2)
	{
		\begin{tikzpicture}[thick,scale=0.4, every node/.style={scale=0.6}]
		\draw[thick] (0,1.7) -- (0,-2.5) -- (1,-2.5) -- (1,1.7);
		\draw (2.5,-.5) node {$R$};
		\draw[thick] (2,1.7) -- (2,0) -- (3,0) -- (3,1.7);
		\draw (0.5,-3) node {$L$};
		\draw[ ->] (4,2.4) -- node[above] {} (2.66,2.4) -- (2.66,1.7);
		\draw (5.0,2.4) node {$2514$};
		\draw[bend left, ->] (2.33,1.7) -- (2.33,2.4) -- node[above] {} (0.66,2.4) -- (0.66,1.7);
		\draw[ <-] (-1.2,2.4) -- node[above] {} (0.33,2.4) -- (0.33,1.7);	
		\draw (0.5,-1.7) node {$3$};
		\draw (2.5,-3.7) node {$2, 3$ never appear together in stack $R$};
		\end{tikzpicture}
	};
	
	\node[below of=root, node distance=3.5cm] (root_3)
	{
		
		\begin{tikzpicture}[thick,scale=0.4, every node/.style={scale=0.6}]
		\draw[thick] (0,1.7) -- (0,-2.5) -- (1,-2.5) -- (1,1.7);
		\draw (2.5,-.5) node {$R$};
		\draw[thick] (2,1.7) -- (2,0) -- (3,0) -- (3,1.7);
		\draw (0.5,-3) node {$L$};
		\draw[ ->] (4,2.4) -- node[above] {} (2.66,2.4) -- (2.66,1.7);
		\draw (4.5,2.4) node {$514$};
		\draw[bend left, ->] (2.33,1.7) -- (2.33,2.4) -- node[above] {} (0.66,2.4) -- (0.66,1.7);
		\draw[ <-] (-1.2,2.4) -- node[above] {} (0.33,2.4) -- (0.33,1.7);	
		\draw (2.5,0.8) node {$2$};
		\draw (2.5,0.3) node {$3$};
	
		\end{tikzpicture}
		
	};

	\node[below of=root_3, node distance=3.3cm] (root_5)
	{
		
		\begin{tikzpicture}[thick,scale=0.4, every node/.style={scale=0.6}]
		\draw[thick] (0,1.7) -- (0,-2.5) -- (1,-2.5) -- (1,1.7);
		\draw (2.5,-.5) node {$R$};
		\draw[thick] (2,1.7) -- (2,0) -- (3,0) -- (3,1.7);
		\draw (0.5,-3) node {$L$};
		\draw[ ->] (4,2.4) -- node[above] {} (2.66,2.4) -- (2.66,1.7);
		\draw (4.5,2.4) node {$514$};
		\draw[bend left, ->] (2.33,1.7) -- (2.33,2.4) -- node[above] {} (0.66,2.4) -- (0.66,1.7);
		\draw[ <-] (-1.2,2.4) -- node[above] {} (0.33,2.4) -- (0.33,1.7);	
		\draw (0.5,-1.7) node {$2$};
		\draw (2.5,0.3) node {$3$};
		\end{tikzpicture}
		
	};
	
	\node[right of=root_5, node distance=3.3cm] (root_7)
	{
		
		\begin{tikzpicture}[thick,scale=0.4, every node/.style={scale=0.6}]
		\draw[thick] (0,1.7) -- (0,-2.5) -- (1,-2.5) -- (1,1.7);
		\draw (2.5,-.5) node {$R$};
		\draw[thick] (2,1.7) -- (2,0) -- (3,0) -- (3,1.7);
		\draw (0.5,-3) node {$L$};
		\draw[ ->] (4,2.4) -- node[above] {} (2.66,2.4) -- (2.66,1.7);
		\draw (4.5,2.4) node {$14$};
		\draw[bend left, ->] (2.33,1.7) -- (2.33,2.4) -- node[above] {} (0.66,2.4) -- (0.66,1.7);
		\draw[ <-] (-1.2,2.4) -- node[above] {} (0.33,2.4) -- (0.33,1.7);	
		\draw (0.5,-1.7) node {$2$};
		\draw (2.5,0.8) node {$5$};
		\draw (2.5,0.3) node {$3$};
		\end{tikzpicture}
		
	};
	\node[right of=root_7, node distance=3.7cm] (root_8)
	{
		
		\begin{tikzpicture}[thick,scale=0.4, every node/.style={scale=0.6}]
		\draw[thick] (0,1.7) -- (0,-2.5) -- (1,-2.5) -- (1,1.7);
		\draw (2.5,-.5) node {$R$};
		\draw[thick] (2,1.7) -- (2,0) -- (3,0) -- (3,1.7);
		\draw (0.5,-3) node {$L$};
		\draw[ ->] (4,2.4) -- node[above] {} (2.66,2.4) -- (2.66,1.7);
		\draw (4.5,2.4) node {$4$};
		\draw (-2.0,2.4) node {$12$};
		\draw[bend left, ->] (2.33,1.7) -- (2.33,2.4) -- node[above] {} (0.66,2.4) -- (0.66,1.7);
		\draw[ <-] (-1.2,2.4) -- node[above] {} (0.33,2.4) -- (0.33,1.7);	
		\draw (0.5,-1.7) node {$5$};
		\draw (2.5,0.3) node {$3$};
		\draw (2.0,-4.0) node {4, 5 must appear together in stack $L$};
		\end{tikzpicture}
		
	};
	
	\node[right of=root_3, node distance=3.3cm] (root_6)
	{
		
		\begin{tikzpicture}[thick,scale=0.4, every node/.style={scale=0.6}]
		\draw[thick] (0,1.7) -- (0,-2.5) -- (1,-2.5) -- (1,1.7);
		\draw (2.5,-.5) node {$R$};
		\draw[thick] (2,1.7) -- (2,0) -- (3,0) -- (3,1.7);
		\draw (0.5,-3) node {$L$};
		\draw[ ->] (4,2.4) -- node[above] {} (2.66,2.4) -- (2.66,1.7);
		\draw (4.5,2.4) node {$14$};
		\draw[bend left, ->] (2.33,1.7) -- (2.33,2.4) -- node[above] {} (0.66,2.4) -- (0.66,1.7);
		\draw[ <-] (-1.2,2.4) -- node[above] {} (0.33,2.4) -- (0.33,1.7);	
		\draw (2.5,1.3) node {$5$};
		\draw (2.5,0.8) node {$2$};
		\draw (2.5,0.3) node {$3$};
		\end{tikzpicture}
		
	};
	\node[right of=root_6,node distance=3.7cm] (root_10)
	{
		
		\begin{tikzpicture}[thick,scale=0.4, every node/.style={scale=0.6}]
		\draw[thick] (0,1.7) -- (0,-2.5) -- (1,-2.5) -- (1,1.7);
		\draw (2.5,-.5) node {$R$};
		\draw[thick] (2,1.7) -- (2,0) -- (3,0) -- (3,1.7);
		\draw (0.5,-3) node {$L$};
		\draw[ ->] (4,2.4) -- node[above] {} (2.66,2.4) -- (2.66,1.7);
		\draw (4.5,2.4) node {$14$};
		\draw[bend left, ->] (2.33,1.7) -- (2.33,2.4) -- node[above] {} (0.66,2.4) -- (0.66,1.7);
		\draw[ <-] (-1.2,2.4) -- node[above] {} (0.33,2.4) -- (0.33,1.7);	
		\draw (0.5,-1.7) node {$5$};
		\draw (2.5,0.8) node {$2$};
		\draw (2.5,0.3) node {$3$};
		\draw (2.0,-4.0) node {4,5  must appear together in stack $L$};
		\end{tikzpicture}
		
	};
	
	\path [line] (root) -- (root_2);
	\path [line] (root) -- (root_3);
	\path [line] (root_3) -- (root_5);
	\path [line] (root_3) -- (root_6);
	\path [line] (root_5) -- (root_7);
	\path [line] (root_7) -- (root_8);
	\path [line] (root_6) -- (root_10);
\end{tikzpicture}
\end{center}
\caption{Sorting $32514$}
\label{fig:lemma4}
\end{figure}
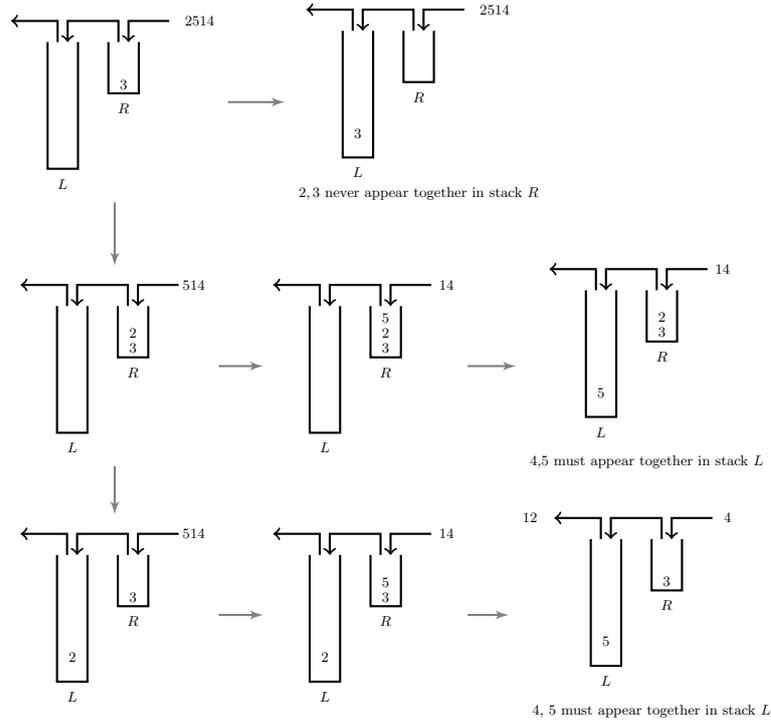

\begin{lemma}
	\label{GZigZag32541}
	Let $\alpha = a_1a_2\dots a_n\in S(t,\infty)$ for $t\geq 3$ and suppose $1\leq i_1<i_2<\dots<i_5\leq n$ with  $a_{i_1}a_{i_2}\dots a_{i_5}$ order-isomorphic to $32541$.  Then, in any sorting process that sorts $\alpha$, if 
	   $a_{i_1},a_{i_2}$ appear together in stack $L$, then  at the step that $a_{i_1}$ is output, \begin{enumerate}\item $a_{i_3},a_{i_4}$ are both in stack  $R$, and \item   if $a_k$ is in stack $L$ then $k<i_2$.\end{enumerate}
\end{lemma}
\begin{proof}
For simplicity let us write  $a_{i_1}=3$, $a_{i_2}=2$, $a_{i_3}=5$, $a_{i_4}=4$, $a_{i_5}=1$, and $\alpha=u_03u_12u_25u_34u_41u_5$.  
Figure~\ref{fig:lemma5} indicates the possible ways to sort these entries. In the case that $2,3$ appear in stack $R$ together, Lemma~\ref{lem:Lrelated} ensures $2,3$ do not appear together in stack $L$. In the other case, before $3$ is moved into stack $L$, any tokens in stack $L$ come from $u_0u_1$.  Thus when $3$ is output the only tokens in stack $L$ will be $a_k$ with $k<i_2$.
Lemma~\ref{lem:Lrelated} ensures that $4,5$ are not placed on top of 3 in stack $L$, so that the step that $3$ is output they sit together in stack $R$.
\qed\end{proof}

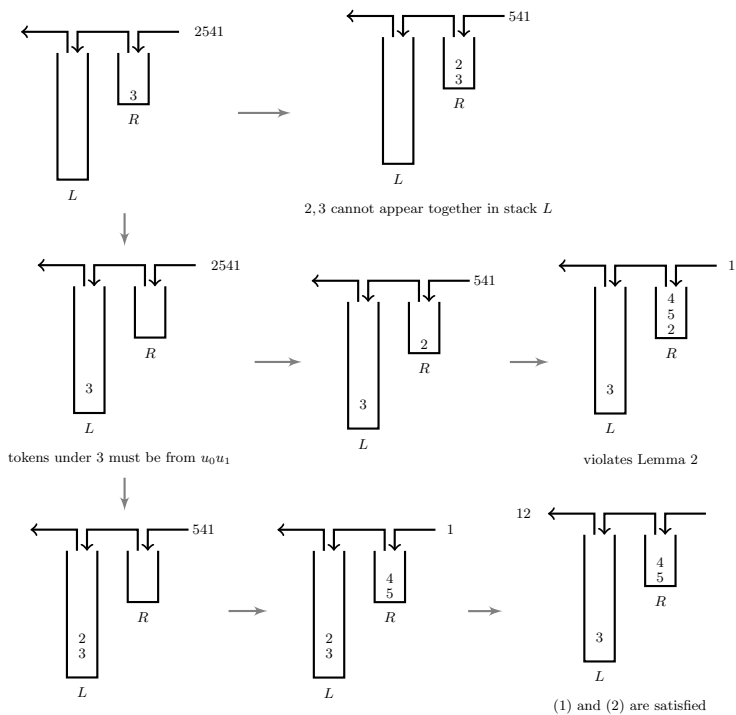
\begin{figure}[ht]
	\begin{center}
		
		\begin{tikzpicture}[scale=2.5, node distance = 2cm, auto]
		\begin{node} (root)
		{
			\begin{tikzpicture}[thick,scale=0.4, every node/.style={scale=0.6}]
			\draw[thick] (0,1.7) -- (0,-2.5) -- (1,-2.5) -- (1,1.7);
			\draw (2.5,-.5) node {$R$};
			\draw[thick] (2,1.7) -- (2,0) -- (3,0) -- (3,1.7);
			\draw (0.5,-3) node {$L$};
			\draw[ ->] (4,2.4) -- node[above] {} (2.66,2.4) -- (2.66,1.7);
			\draw (5.0,2.4) node {$2541$};
			\draw[bend left, ->] (2.33,1.7) -- (2.33,2.4) -- node[above] {} (0.66,2.4) -- (0.66,1.7);
			\draw[ <-] (-1.2,2.4) -- node[above] {} (0.33,2.4) -- (0.33,1.7);	
			\draw (2.5,0.3) node {$3$};
			\end{tikzpicture}
		};
	\end{node}
	
	\node[right of=root, node distance=4.0cm] (root_2)
	{
		\begin{tikzpicture}[thick,scale=0.4, every node/.style={scale=0.6}]
		\draw[thick] (0,1.7) -- (0,-2.5) -- (1,-2.5) -- (1,1.7);
		\draw (2.5,-.5) node {$R$};
		\draw[thick] (2,1.7) -- (2,0) -- (3,0) -- (3,1.7);
		\draw (0.5,-3) node {$L$};
		\draw[ ->] (4,2.4) -- node[above] {} (2.66,2.4) -- (2.66,1.7);
		\draw (4.5,2.4) node {$541$};
		\draw[bend left, ->] (2.33,1.7) -- (2.33,2.4) -- node[above] {} (0.66,2.4) -- (0.66,1.7);
		\draw[ <-] (-1.2,2.4) -- node[above] {} (0.33,2.4) -- (0.33,1.7);	
		\draw (2.5,0.8) node {$2$};
		\draw (2.5,0.3) node {$3$};
		\draw (1.5,-4.0) node {$2,3$ cannot appear together in stack $L$};
		\end{tikzpicture}
	};

	\node[below of=root, node distance=3.3cm] (root_3)
	{
		
		\begin{tikzpicture}[thick,scale=0.4, every node/.style={scale=0.6}]
		\draw[thick] (0,1.7) -- (0,-2.5) -- (1,-2.5) -- (1,1.7);
		\draw (2.5,-.5) node {$R$};
		\draw[thick] (2,1.7) -- (2,0) -- (3,0) -- (3,1.7);
		\draw (0.5,-3) node {$L$};
		\draw[ ->] (4,2.4) -- node[above] {} (2.66,2.4) -- (2.66,1.7);
		\draw (5.0,2.4) node {$2541$};
		\draw[bend left, ->] (2.33,1.7) -- (2.33,2.4) -- node[above] {} (0.66,2.4) -- (0.66,1.7);
		\draw[ <-] (-1.2,2.4) -- node[above] {} (0.33,2.4) -- (0.33,1.7);	
		\draw (0.5,-1.7) node {$3$};
		\draw (1.5,-4.0) node {tokens under $3$ must be from $u_0u_1$};
		\end{tikzpicture}
		
	};
	\node[right of=root_3, node distance=3.7cm] (root_6)
	{
		
		\begin{tikzpicture}[thick,scale=0.4, every node/.style={scale=0.6}]
		\draw[thick] (0,1.7) -- (0,-2.5) -- (1,-2.5) -- (1,1.7);
		\draw (2.5,-.5) node {$R$};
		\draw[thick] (2,1.7) -- (2,0) -- (3,0) -- (3,1.7);
		\draw (0.5,-3) node {$L$};
		\draw[ ->] (4,2.4) -- node[above] {} (2.66,2.4) -- (2.66,1.7);
		\draw (4.5,2.4) node {$541$};
		\draw[bend left, ->] (2.33,1.7) -- (2.33,2.4) -- node[above] {} (0.66,2.4) -- (0.66,1.7);
		\draw[ <-] (-1.2,2.4) -- node[above] {} (0.33,2.4) -- (0.33,1.7);	
		\draw (2.5,0.3) node {$2$};
		\draw (0.5,-1.7) node {$3$};
		\end{tikzpicture}
		
	};
	\node[right of=root_6,node distance=3.2cm] (root_10)
	{
		
		\begin{tikzpicture}[thick,scale=0.4, every node/.style={scale=0.6}]
		\draw[thick] (0,1.7) -- (0,-2.5) -- (1,-2.5) -- (1,1.7);
		\draw (2.5,-.5) node {$R$};
		\draw[thick] (2,1.7) -- (2,0) -- (3,0) -- (3,1.7);
		\draw (0.5,-3) node {$L$};
		\draw[ ->] (4,2.4) -- node[above] {} (2.66,2.4) -- (2.66,1.7);
		\draw (4.5,2.4) node {$1$};
		\draw[bend left, ->] (2.33,1.7) -- (2.33,2.4) -- node[above] {} (0.66,2.4) -- (0.66,1.7);
		\draw[ <-] (-1.2,2.4) -- node[above] {} (0.33,2.4) -- (0.33,1.7);	
		\draw (0.5,-1.7) node {$3$};
		\draw (2.5,1.3) node {$4$};
		\draw (2.5,0.8) node {$5$};
		\draw (2.5,0.3) node {$2$};
		\draw (1.5,-4.0) node {violates Lemma~\ref{lem:sbm}};
		\end{tikzpicture}
		
	};
	
	\node[below of=root_3, node distance=3.3cm] (root_5)
	{
		
		\begin{tikzpicture}[thick,scale=0.4, every node/.style={scale=0.6}]
		\draw[thick] (0,1.7) -- (0,-2.5) -- (1,-2.5) -- (1,1.7);
		\draw (2.5,-.5) node {$R$};
		\draw[thick] (2,1.7) -- (2,0) -- (3,0) -- (3,1.7);
		\draw (0.5,-3) node {$L$};
		\draw[ ->] (4,2.4) -- node[above] {} (2.66,2.4) -- (2.66,1.7);
		\draw (4.5,2.4) node {$541$};
		\draw[bend left, ->] (2.33,1.7) -- (2.33,2.4) -- node[above] {} (0.66,2.4) -- (0.66,1.7);
		\draw[ <-] (-1.2,2.4) -- node[above] {} (0.33,2.4) -- (0.33,1.7);	
		\draw (0.5,-1.2) node {$2$};
		\draw (0.5,-1.7) node {$3$};
		\end{tikzpicture}
		
	};
	
	\node[right of=root_5, node distance=3.2cm] (root_7)
	{
		
		\begin{tikzpicture}[thick,scale=0.4, every node/.style={scale=0.6}]
		\draw[thick] (0,1.7) -- (0,-2.5) -- (1,-2.5) -- (1,1.7);
		\draw (2.5,-.5) node {$R$};
		\draw[thick] (2,1.7) -- (2,0) -- (3,0) -- (3,1.7);
		\draw (0.5,-3) node {$L$};
		\draw[ ->] (4,2.4) -- node[above] {} (2.66,2.4) -- (2.66,1.7);
		\draw (4.5,2.4) node {$1$};
		\draw[bend left, ->] (2.33,1.7) -- (2.33,2.4) -- node[above] {} (0.66,2.4) -- (0.66,1.7);
		\draw[ <-] (-1.2,2.4) -- node[above] {} (0.33,2.4) -- (0.33,1.7);	
		\draw (0.5,-1.2) node {$2$};
		\draw (0.5,-1.7) node {$3$};
		\draw (2.5,0.8) node {$4$};
		\draw (2.5,0.3) node {$5$};
		\end{tikzpicture}
		
	};
	\node[right of=root_7, node distance=3.3cm] (root_8)
	{
		
		\begin{tikzpicture}[thick,scale=0.4, every node/.style={scale=0.6}]
		\draw[thick] (0,1.7) -- (0,-2.5) -- (1,-2.5) -- (1,1.7);
		\draw (2.5,-.5) node {$R$};
		\draw[thick] (2,1.7) -- (2,0) -- (3,0) -- (3,1.7);
		\draw (0.5,-3) node {$L$};
		\draw[ ->] (4,2.4) -- node[above] {} (2.66,2.4) -- (2.66,1.7);
		\draw (4.5,2.4) node {};
		\draw (-2.0,2.4) node {$12$};
		\draw[bend left, ->] (2.33,1.7) -- (2.33,2.4) -- node[above] {} (0.66,2.4) -- (0.66,1.7);
		\draw[ <-] (-1.2,2.4) -- node[above] {} (0.33,2.4) -- (0.33,1.7);	
		\draw (0.5,-1.7) node {$3$};
		\draw (2.5,0.8) node {$4$};
		\draw (2.5,0.3) node {$5$};
		\draw (1.5,-4.0) node {(1) and (2) are satisfied};
		\end{tikzpicture}
		
	};
	
	\path [line] (root) -- (root_2);
	\path [line] (root) -- (root_3);
	\path [line] (root_3) -- (root_5);
	\path [line] (root_3) -- (root_6);
	\path [line] (root_5) -- (root_7);
	\path [line] (root_7) -- (root_8);
	\path [line] (root_6) -- (root_10);

\end{tikzpicture}
\end{center}
\caption{Sorting $32541$}
\label{fig:lemma5}
\end{figure}

\section{An infinite antichain}
We use the following notation. If $\alpha=a_1\dots a_n$ is a permutation of $12\dots n$ and $m\in\mathbb Z$ then let $\alpha_m$
be the permutation obtained by adding $m$ to each entry of $\alpha$. For example $(1\ 2\ 3)_4=5\ 6\ 7$ and $13_6=19$.

We construct a family of permutations $\mathscr G=\{G_{i}\mid i\in\N\}$ as follows. Define 
\[\begin{array}{rcl}P &=& 2 \ 4  \ 3 \ 7 \ 6 \ 1 \\
x_j &=& (10 \ 5\  9)_{6j}  \\
y_j &=& (13 \ 12\  8)_{6j}\\
S_i &=&  (14 \ 15 \ 11)_{6i} \\
G_i &=& P \ x_0 \ y_ 0 \ x_1\  y_1 \ \dots \ x_i \ y_i \ S_i \end{array}\]
The first three terms are 
\[\begin{array}{l}G_0= 2  \  4 \   3 \   7 \   6 \   1  \   {(10\  5\ 9)\  (13\  12\  8)}\   14  \  15  \  11,  \\ 
G_1= 2  \  4 \   3 \   7 \   6 \   1  \   {(10\  5\   9)\   (13\  12\  8)\  (16\   11\  15)\  (19\  18\  14)} \  20 \  21 \  17,  \\ 
G_2= 
P\  {(10\  5\  9)\  (13\  12\  8)\  (16\  11\  15) \  (19 \  18\  14)\  (22\  17\  21)} 
(25\  24\  20) \ 26 \  27 \  23.\end{array}\] 
A diagram of $G_2$ is shown in Figure~\ref{fig:diag} which shows the general pattern.

  \begin{figure}[ht]

\begin{center}
      \begin{tikzpicture}[scale = .4, auto = center, inner sep=.3mm]

   \draw[color=gray] (0.5,0.5) -- (0.5,7.5) -- (6.5,7.5) -- (6.5,0.5) -- (0.5,0.5);

     \draw[color=gray] (6.5,4.5) -- (9.5,4.5) -- (9.5,10.5) -- (6.5,10.5) -- (6.5,4.5);
     \draw[color=gray] (9.5,7.5) -- (12.5,7.5) -- (12.5,13.5) -- (9.5,13.5) -- (9.5,7.5);
     
     \draw[color=gray] (9.5,7.5) -- (12.5,7.5) -- (12.5,13.5) -- (9.5,13.5) -- (9.5,7.5);
     \draw[color=gray] (12.5,10.5) -- (15.5,10.5) -- (15.5,16.5) -- (12.5,16.5) -- (12.5,10.5);

     \draw[color=gray] (12.5,10.5) -- (15.5,10.5) -- (15.5,16.5) -- (12.5,16.5) -- (12.5,10.5);
     \draw[color=gray] (15.5,13.5) -- (18.5,13.5) -- (18.5,19.5) -- (15.5,19.5) -- (15.5,13.5);
 
     \draw[color=gray] (15.5,13.5) -- (18.5,13.5) -- (18.5,19.5) -- (15.5,19.5) -- (15.5,13.5);
     \draw[color=gray] (18.5,16.5) -- (21.5,16.5) -- (21.5,22.5) -- (18.5,22.5) -- (18.5,16.5);
            
     \draw[color=gray] (18.5,16.5) -- (21.5,16.5) -- (21.5,22.5) -- (18.5,22.5) -- (18.5,16.5);
     \draw[color=gray] (21.5,19.5) -- (24.5,19.5) -- (24.5,25.5) -- (21.5,25.5) -- (21.5,19.5);

     \draw[color=gray] (24.5,22.5) -- (27.5,22.5) -- (27.5,27.5) -- (24.5,27.5) -- (24.5,22.5);
                
                   \node[circle, fill=black] at (1,2) {};
        \node[circle, fill=black] at (2,4) {};
        \node[circle, fill=black] at (3,3) {};
        \node[circle, fill=black] at (4,7) {};
        \node[circle, fill=black] at (5,6) {};
        \node[circle, fill=black] at (6,1) {};
        
        \node[circle, fill=black] at (7,10) {};
          \node[circle, fill=black] at (8,5) {};
            \node[circle, fill=black] at (9,9) {};
            
                    \node[circle, fill=black] at (10,13) {};               
                       \node[circle, fill=black] at (11,12) {};
        \node[circle, fill=black] at (12,8) {};
        
        \node[circle, fill=black] at (13,16) {};
        \node[circle, fill=black] at (14,11) {};
        \node[circle, fill=black] at (15,15) {};
        
                \node[circle, fill=black] at (16,19) {};
        \node[circle, fill=black] at (17,18) {};
          \node[circle, fill=black] at (18,14) {};

            \node[circle, fill=black] at (19,22) {};
                    \node[circle, fill=black] at (20,17) {};       
           \node[circle, fill=black] at (21,21) {};
           
                \node[circle, fill=black] at (22,25) {};
        \node[circle, fill=black] at (23,24) {};
        \node[circle, fill=black] at (24,20) {};
        
        \node[circle, fill=black] at (25,26) {};
        \node[circle, fill=black] at (26,27) {};
        \node[circle, fill=black] at (27,23) {};


      \end{tikzpicture}
      \end{center}

\caption{Diagram of the permutation $G_2= 
2 \  4  \  3  \  7 \  6 \  1 \   
x_0\ y_0\ x_1\ y_1\ x_2 \ y_2 \ 26 \  27 \  23$ }
\label{fig:diag}
  \end{figure}
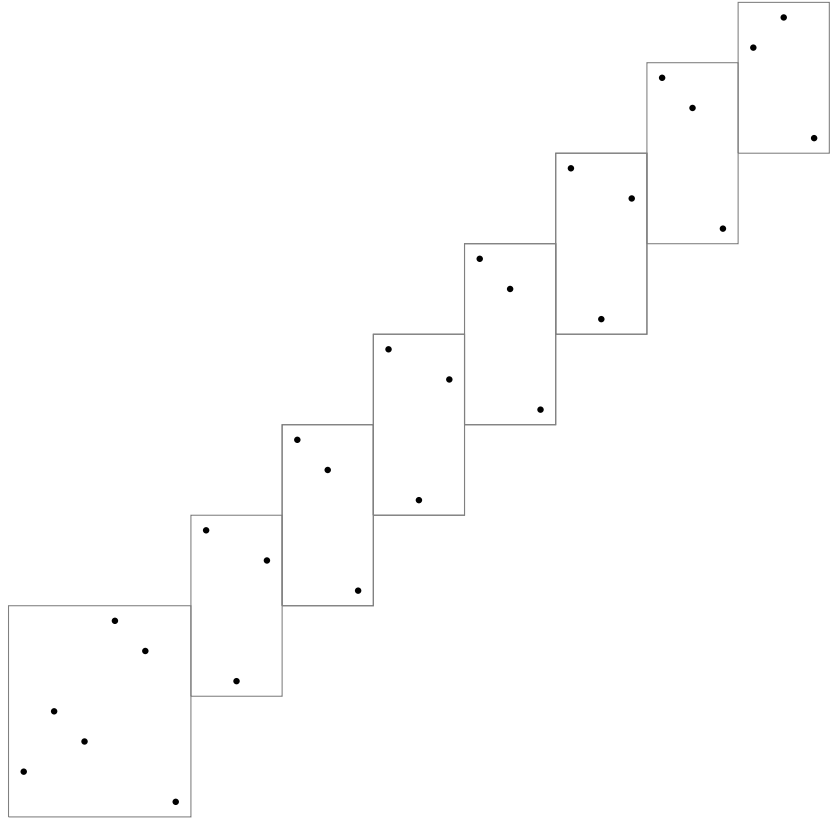

We will prove that each $G_i$  is an element of the basis of $S(3,\infty)$ for all $i\in \N$.
Note that if we define $x_{-1}, y_{-1}$ to be empty,
 $G_{-1}=2    4    3    7    6    1     8    9    5$ is also an element of the basis. We noticed this and $G_0$ had a particular pattern which we could extend using $x_jy_j$. However,  we exclude $G_{-1}$ from our antichain to make the proofs simpler. 
 
\begin{proposition}
	\label{prop:Gi} The permutation
	$G_i\not\in S(3,\infty)$  for all $i\in \mathbb N$.
\end{proposition}
\begin{proof} 
Suppose for contradiction that $G_i$ can be sorted by some sorting process. Since $P$ is order isomorphic to $243651$, by 
 Lemma~\ref{lem:67R} in any sorting process $7,6$ appear together in  stack $R$. 
Next,  $7\ 6\ 10 \ 5 \ 9$ is order isomorphic to $32514$   so 
 by Lemma~\ref{GZigZag32514} since $7,6$ appear together in stack $R$ we must have that $10,9$ appear together in stack $L$ at some point in the process.
 
 Now consider $x_jy_j=(10 \ 5 \ 9 \ 13\ 12 \ 8)_{6j}$, and assume that $10_{6j},9_{6j}$ both appear in stack $L$ together. Since 
$ (10 \ 9  \ 13\ 12 \ 8)_{6j}$ is order isomorphic to $32541$ by Lemma~\ref{GZigZag32541}
$13_{6j},12_{6j}$ must be placed together in stack $R$ and stay there until $10_{6j}$ is output.

 Next consider $y_jx_{j+1}=( 13\ 12 \ 8 \ 16 \ 11 \ 15)_{6j}$, and assume that $13_{6j},12_{6j}$ both appear in stack $R$ together. Then since 
$ (13\ 12  \ 16 \ 11 \ 15)_{6j}$ is order isomorphic to $32514$ by Lemma~\ref{GZigZag32514}
we have $16_{6j},15_{6j}$ appear together in stack $L$. Note that $16_{6j},15_{6j}=10_{6(j+1)},9_{6(j+1)}$, so putting the above observations together we see that for all $0\leq j\leq i$ we have $10_{6j},9_{6j}$ both appear in stack $L$ together and 
$13_{6j},12_{6j}$ appear together in stack $R$ and stay there until $10_{6j}$ is output.

Now we consider the suffix $$x_iy_iS_i= (10 \ 5 \ 9 \ 13 \ 12\  8\ 14 \ 15 \ 11)_{6i}$$ where $10_{6i},9_{6i}$ are together in stack $L$. Lemma~\ref{GZigZag32541} tells us not only that 
$13_{6i},12_{6i}$ appear together in stack $R$ and stay there until $10_{6i}$ is output, but that anything sitting underneath  $10_{6i}$  in stack $L$ comes {\em before} $9_{6i}$ in $G_i$, so in particular $14_{6i}, 15_{6i}$  are not underneath $10_{6i}$. All possible processes to sort $x_iy_iS$ are shown in Fig.~\ref{fig:moment}. All possible sorting moves fail, which means $G_i$ cannot be sorted.
\qed\end{proof}
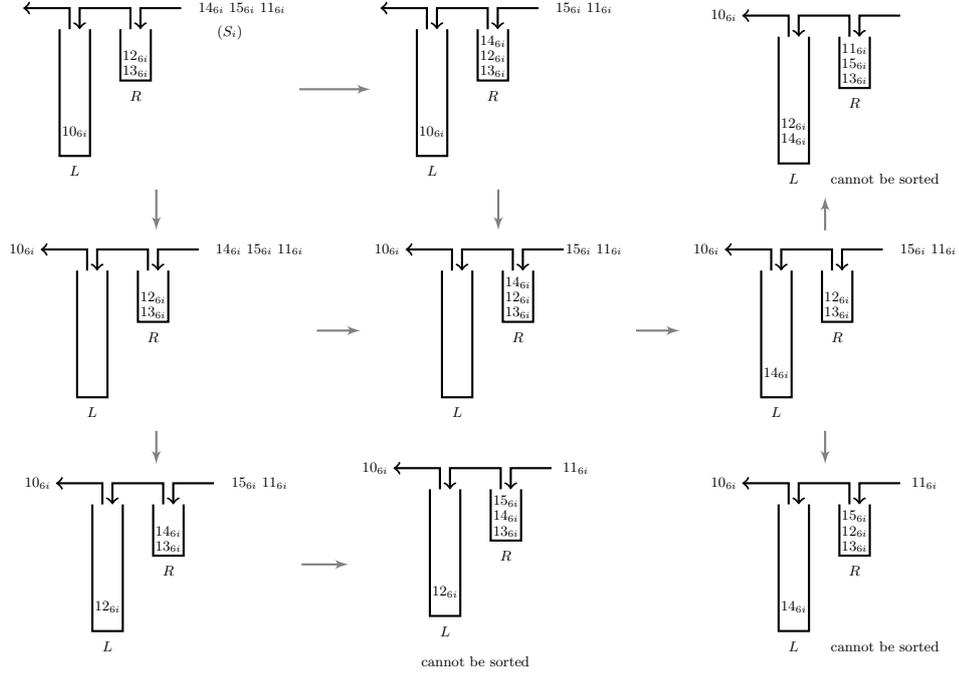
\begin{figure}[ht]
	\begin{center}
		
		\begin{tikzpicture}[scale=2, node distance = 2.5cm, auto]
		
		\node(root)
		{
			\begin{tikzpicture}[thick,scale=0.4, every node/.style={scale=0.6}]
			\draw[thick] (0,1.7) -- (0,-2.5) -- (1,-2.5) -- (1,1.7);
			\draw (2.5,-.5) node {$R$};
			\draw[thick] (2,1.7) -- (2,0) -- (3,0) -- (3,1.7);
			\draw (0.5,-3) node {$L$};
			\draw[ ->] (4,2.4) -- node[above] {} (2.66,2.4) -- (2.66,1.7);
			\draw (6.0,2.4) node {$14_{6i}$ $15_{6i}$ $11_{6i}$};
			\draw[bend left, ->] (2.33,1.7) -- (2.33,2.4) -- node[above] {} (0.66,2.4) -- (0.66,1.7);
			\draw[ <-] (-1.2,2.4) -- node[above] {} (0.33,2.4) -- (0.33,1.7);	
			\draw (2.5,0.8) node {$12_{6i}$};
			\draw (2.5,0.3) node {$13_{6i}$};
			\draw (0.5,-1.7) node {$10_{6i}$};

			\draw (5.6,1.6) node {$(S_i)$};
			
			\end{tikzpicture}
		};

		\node[right of=root, node distance=4.5cm] (root_2)
		{
			\begin{tikzpicture}[thick,scale=0.4, every node/.style={scale=0.6}]
			\draw[thick] (0,1.7) -- (0,-2.5) -- (1,-2.5) -- (1,1.7);
			\draw (2.5,-.5) node {$R$};
			\draw[thick] (2,1.7) -- (2,0) -- (3,0) -- (3,1.7);
			\draw (0.5,-3) node {$L$};
			\draw[ ->] (4,2.4) -- node[above] {} (2.66,2.4) -- (2.66,1.7);
			\draw (5.5,2.4) node {$15_{6i}$ $11_{6i}$};
			\draw[bend left, ->] (2.33,1.7) -- (2.33,2.4) -- node[above] {} (0.66,2.4) -- (0.66,1.7);
			\draw[ <-] (-1.2,2.4) -- node[above] {} (0.33,2.4) -- (0.33,1.7);	
			\draw (2.5,1.3) node {$14_{6i}$};
			\draw (2.5,0.8) node {$12_{6i}$};
			\draw (2.5,0.3) node {$13_{6i}$};
			\draw (0.5,-1.7) node {$10_{6i}$};
			\end{tikzpicture}
		};
		
		\node[below of=root, node distance=3.2cm] (root_3)
		{
			\begin{tikzpicture}[thick,scale=0.4, every node/.style={scale=0.6}]
			\draw[thick] (0,1.7) -- (0,-2.5) -- (1,-2.5) -- (1,1.7);
			\draw (2.5,-.5) node {$R$};
			\draw[thick] (2,1.7) -- (2,0) -- (3,0) -- (3,1.7);
			\draw (0.5,-3) node {$L$};
			\draw[ ->] (4,2.4) -- node[above] {} (2.66,2.4) -- (2.66,1.7);
			\draw (6.0,2.4) node {$14_{6i}$ $15_{6i}$ $11_{6i}$};
			\draw (-1.8,2.4) node {$10_{6i}$};
			\draw[bend left, ->] (2.33,1.7) -- (2.33,2.4) -- node[above] {} (0.66,2.4) -- (0.66,1.7);
			\draw[ <-] (-1.2,2.4) -- node[above] {} (0.33,2.4) -- (0.33,1.7);	
			\draw (2.5,0.8) node {$12_{6i}$};
			\draw (2.5,0.3) node {$13_{6i}$};
			\end{tikzpicture}
		};
		
		\node[below of=root_3, node distance=3.1cm] (root_4)
		{
			\begin{tikzpicture}[thick,scale=0.4, every node/.style={scale=0.6}]
			\draw[thick] (0,1.7) -- (0,-2.5) -- (1,-2.5) -- (1,1.7);
			\draw (2.5,-.5) node {$R$};
			\draw[thick] (2,1.7) -- (2,0) -- (3,0) -- (3,1.7);
			\draw (0.5,-3) node {$L$};
			\draw[ ->] (4,2.4) -- node[above] {} (2.66,2.4) -- (2.66,1.7);
			\draw (5.5,2.4) node {$15_{6i}$ $11_{6i}$};
			\draw (-1.8,2.4) node {$10_{6i}$};
			\draw[bend left, ->] (2.33,1.7) -- (2.33,2.4) -- node[above] {} (0.66,2.4) -- (0.66,1.7);
			\draw[ <-] (-1.2,2.4) -- node[above] {} (0.33,2.4) -- (0.33,1.7);	
			\draw (2.5,0.8) node {$14_{6i}$};
			\draw (2.5,0.3) node {$13_{6i}$};
			\draw (0.5,-1.7) node {$12_{6i}$};
			\end{tikzpicture}
		};
		
		\node[right of=root_4, node distance=4.2cm] (root_6)
		{
			\begin{tikzpicture}[thick,scale=0.4, every node/.style={scale=0.6}]
			\draw[thick] (0,1.7) -- (0,-2.5) -- (1,-2.5) -- (1,1.7);
			\draw (2.5,-.5) node {$R$};
			\draw[thick] (2,1.7) -- (2,0) -- (3,0) -- (3,1.7);
			\draw (0.5,-3) node {$L$};
			\draw[ ->] (4,2.4) -- node[above] {} (2.66,2.4) -- (2.66,1.7);
			\draw (4.8,2.4) node {$11_{6i}$};
			\draw (-1.8,2.4) node {$10_{6i}$};
			\draw[bend left, ->] (2.33,1.7) -- (2.33,2.4) -- node[above] {} (0.66,2.4) -- (0.66,1.7);
			\draw[ <-] (-1.2,2.4) -- node[above] {} (0.33,2.4) -- (0.33,1.7);	
			\draw (2.5,1.3) node {$15_{6i}$};
			\draw (2.5,0.8) node {$14_{6i}$};
			\draw (2.5,0.3) node {$13_{6i}$};
			\draw (0.5,-1.7) node {$12_{6i}$};
			\draw (1.5,-4.0) node {cannot be sorted};
			\end{tikzpicture}
		};
		
		\node[right of=root_3, node distance=4.5cm] (root_5)
		{
			\begin{tikzpicture}[thick,scale=0.4, every node/.style={scale=0.6}]
			\draw[thick] (0,1.7) -- (0,-2.5) -- (1,-2.5) -- (1,1.7);
			\draw (2.5,-.5) node {$R$};
			\draw[thick] (2,1.7) -- (2,0) -- (3,0) -- (3,1.7);
			\draw (0.5,-3) node {$L$};
			\draw[ ->] (4,2.4) -- node[above] {} (2.66,2.4) -- (2.66,1.7);
			\draw (5.0,2.4) node {$15_{6i}$ $11_{6i}$};
			\draw (-1.8,2.4) node {$10_{6i}$};
			\draw[bend left, ->] (2.33,1.7) -- (2.33,2.4) -- node[above] {} (0.66,2.4) -- (0.66,1.7);
			\draw[ <-] (-1.2,2.4) -- node[above] {} (0.33,2.4) -- (0.33,1.7);	
			\draw (2.5,1.3) node {$14_{6i}$};
			\draw (2.5,0.8) node {$12_{6i}$};
			\draw (2.5,0.3) node {$13_{6i}$};
			\end{tikzpicture}
		};
		
		\node[right of=root_5, node distance=4.3cm] (root_8)
		{
			\begin{tikzpicture}[thick,scale=0.4, every node/.style={scale=0.6}]
			\draw[thick] (0,1.7) -- (0,-2.5) -- (1,-2.5) -- (1,1.7);
			\draw (2.5,-.5) node {$R$};
			\draw[thick] (2,1.7) -- (2,0) -- (3,0) -- (3,1.7);
			\draw (0.5,-3) node {$L$};
			\draw[ ->] (4,2.4) -- node[above] {} (2.66,2.4) -- (2.66,1.7);
			\draw (5.5,2.4) node {$15_{6i}$ $11_{6i}$};
			\draw (-1.8,2.4) node {$10_{6i}$};
			\draw[bend left, ->] (2.33,1.7) -- (2.33,2.4) -- node[above] {} (0.66,2.4) -- (0.66,1.7);
			\draw[ <-] (-1.2,2.4) -- node[above] {} (0.33,2.4) -- (0.33,1.7);	
			\draw (2.5,0.8) node {$12_{6i}$};
			\draw (2.5,0.3) node {$13_{6i}$};
			\draw (0.5,-1.7) node {$14_{6i}$};
			\end{tikzpicture}
		};
		
		\node[below of=root_8, node distance=3.1cm] (root_9)
		{
			\begin{tikzpicture}[thick,scale=0.4, every node/.style={scale=0.6}]
			\draw[thick] (0,1.7) -- (0,-2.5) -- (1,-2.5) -- (1,1.7);
			\draw (2.5,-.5) node {$R$};
			\draw[thick] (2,1.7) -- (2,0) -- (3,0) -- (3,1.7);
			\draw (0.5,-3) node {$L$};
			\draw[ ->] (4,2.4) -- node[above] {} (2.66,2.4) -- (2.66,1.7);
			\draw (4.8,2.4) node {$11_{6i}$};
			\draw (-1.8,2.4) node {$10_{6i}$};
			\draw[bend left, ->] (2.33,1.7) -- (2.33,2.4) -- node[above] {} (0.66,2.4) -- (0.66,1.7);
			\draw[ <-] (-1.2,2.4) -- node[above] {} (0.33,2.4) -- (0.33,1.7);	
			\draw (2.5,1.3) node {$15_{6i}$};
			\draw (2.5,0.8) node {$12_{6i}$};
			\draw (2.5,0.3) node {$13_{6i}$};
			\draw (0.5,-1.7) node {$14_{6i}$};
			\draw (3.5,-3.0) node {cannot be sorted};
			\end{tikzpicture}
		};
		
		\node[above of=root_8, node distance=3.1cm] (root_10)
		{
			\begin{tikzpicture}[thick,scale=0.4, every node/.style={scale=0.6}]
			\draw[thick] (0,1.7) -- (0,-2.5) -- (1,-2.5) -- (1,1.7);
			\draw (2.5,-.5) node {$R$};
			\draw[thick] (2,1.7) -- (2,0) -- (3,0) -- (3,1.7);
			\draw (0.5,-3) node {$L$};
			\draw[ ->] (4,2.4) -- node[above] {} (2.66,2.4) -- (2.66,1.7);
			\draw (4.8,2.4) node {};
			\draw (-1.8,2.4) node {$10_{6i}$};
			\draw[bend left, ->] (2.33,1.7) -- (2.33,2.4) -- node[above] {} (0.66,2.4) -- (0.66,1.7);
			\draw[ <-] (-1.2,2.4) -- node[above] {} (0.33,2.4) -- (0.33,1.7);	
			\draw (2.5,1.3) node {$11_{6i}$};
			\draw (2.5,0.8) node {$15_{6i}$};
			\draw (2.5,0.3) node {$13_{6i}$};
			\draw (0.5,-1.2) node {$12_{6i}$};
			\draw (0.5,-1.7) node {$14_{6i}$};
			\draw (3.5,-3.0) node {cannot be sorted};
			\end{tikzpicture}
		};		
		
		\path [line] (root) -- (root_2);
		\path [line] (root) -- (root_3);
		\path [line] (root_3) -- (root_4);
		\path [line] (root_3) -- (root_5);
		\path [line] (root_2) -- (root_5);
		\path [line] (root_4) -- (root_6);
		\path [line] (root_5) -- (root_8);
		\path [line] (root_8) -- (root_9);
		\path [line] (root_8) -- (root_10);
	\end{tikzpicture}
	
	\caption{All possible ways to sort  $x_iy_iS$}
	\label{fig:moment}
\end{center}
\end{figure}

The idea of the preceding proof can be summarised informally as follows. The prefix $P$  forces $7,6$ to be together in stack $R$, then Lemmas 4 and 5 alternately imply that the $10_{6j},9_{6j}$ terms of $x_j$ must be in stack $L$ and the $13_{6j},12_{6j}$ terms of $y_j$ must be in stack $R$. When we reach the suffix $S_i$ the fact that certain entries are forced to be in a particular stack means we are unable to sort the final terms.
We now show that if a single entry is removed from $G_i$, we can choose to place the $10_{6j},9_{6j}$ terms in stack $R$ and  $13_{6j},12_{6j}$ terms in stack $L$, which allows the suffix to be sorted.

\begin{lemma} \label{lem:can_sort1}
Let $0\leq j\leq i$. If  stack $R$ contains one or both of $10_{6j}, 9_{6j}$ in ascending order,  and $y_j\dots y_iS_i$ is to be input as in Fig.~\ref{fig:Premoved2}, 
then there is a sorting procedure to output all remaining entries in order.
\end{lemma}
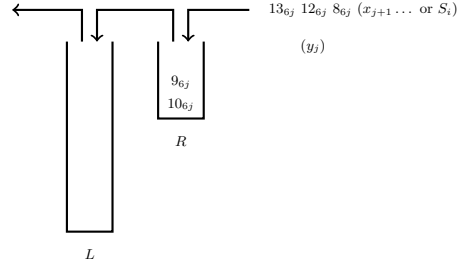
\begin{figure}[ht]
	\begin{center}
		\begin{tikzpicture}[thick,scale=0.6, every node/.style={scale=0.6}]
		\draw[thick] (0,1.7) -- (0,-2.5) -- (1,-2.5) -- (1,1.7);
		\draw (2.5,-.5) node {$R$};
		\draw[thick] (2,1.7) -- (2,0) -- (3,0) -- (3,1.7);
		\draw (0.5,-3) node {$L$};
		\draw[ ->] (4,2.4) -- node[above] {} (2.66,2.4) -- (2.66,1.7);

		\draw (6.5,2.4) node {$13_{6j} \ 12_{6j} \ 8_{6j}  \ (x_{j+1}\dots \text{ or } S_i)$};
		\draw[bend left, ->] (2.33,1.7) -- (2.33,2.4) -- node[above] {} (0.66,2.4) -- (0.66,1.7);
		\draw[ <-] (-1.2,2.4) -- node[above] {} (0.33,2.4) -- (0.33,1.7);
		
		\draw (5.4,1.6) node {$(y_{j})$};
		
		\draw (2.5,.8) node {$9_{6j}$};
		\draw (2.5,0.3) node {$10_{6j}$};
		\end{tikzpicture}
		\caption{A sortable configuration}
		\label{fig:Premoved2}
	\end{center}
\end{figure}
\begin{proof}
For $j<i$ 
 move $13_{6j}, 12_{6j}$ into stack $L$, output $8_{6j}, 9_{6j},10_{6j}$, move $16_{6j}=10_{6(j+1)}$ into stack $R$, output $11_{6j}=5_{6(j+1)}$, output  $13_{6j}, 12_{6j}$ from stack $L$ and input $15_{6j}=9_{6(j+1)}$ so that the configuration has the same form as Fig.~\ref{fig:Premoved2} with $j$ incremented by 1.

For $j=i$  the remaining input is $(13 \ 12 \ 8 \  14 \ 15 \ 11)_{6j}$.   Put $13_{6i},12_{6i} $ in stack $L$ in order,  output $8_{6i}, 9_{6i}, 10_{6i}$, put $  14_{6i},15_{6i}$ in stack $R$ and output $11_{6i}$, $12_{6i},13_{6i}$, move $15_{6i}$ into stack $L$ and output $14_{6i}$ then $15_{6i}$.

If one of $ 9_{6j},10_{6j}$ is missing, use the same procedure ignoring the missing entry.
\qed\end{proof}

\begin{lemma} \label{lem:can_sort2}
Let $0\leq j\leq i$. 
If stack $L$ contains one or both of $12_{6j}, 13_{6j}$ in ascending order, and $x_{j+1}\dots S_i$ (or just $S_i$ if $j=i$) is to be input as in Fig.~\ref{fig:67_2}, then there is a sorting procedure to output all remaining entries in order.
\end{lemma}

			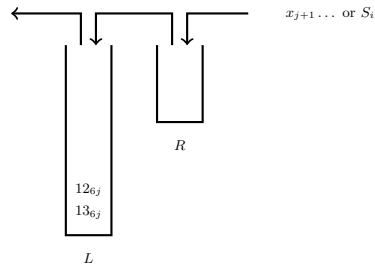
\begin{figure}[ht]
	\begin{center}
		\begin{tikzpicture}[thick,scale=0.6, every node/.style={scale=0.6}]
			\draw[thick] (0,1.7) -- (0,-2.5) -- (1,-2.5) -- (1,1.7);
			\draw (2.5,-.5) node {$R$};
		\draw[thick] (2,1.7) -- (2,0) -- (3,0) -- (3,1.7);
		\draw (0.5,-3) node {$L$};
		\draw[ ->] (4,2.4) -- node[above] {} (2.66,2.4) -- (2.66,1.7);

		\draw (5.8,2.4) node {$x_{j+1} \dots \text{ or } S_i$};
		\draw[bend left, ->] (2.33,1.7) -- (2.33,2.4) -- node[above] {} (0.66,2.4) -- (0.66,1.7);
		\draw[ <-] (-1.2,2.4) -- node[above] {} (0.33,2.4) -- (0.33,1.7);

		\draw (0.5,-1.5) node {$12_{6j}$};
		\draw (0.5,-2) node {$13_{6j}$};
		\end{tikzpicture}
		\caption{Another sortable configuration} 
		\label{fig:67_2}
	\end{center}
\end{figure}

\begin{proof}
If $j<i$
 move $10_{6(j+1)}$ into stack $R$, output $5_{6(j+1)}, 12_{6j},13_{6j}$, move $9_{6(j+1)}$ to stack $R$ to reach the  configuration in  Fig.~\ref{fig:Premoved2}, which we can sort by Lemma~\ref{lem:can_sort1}.
 If $j=i$ then the remaining input is just $S_i=(14\ 15 \ 11)_{6i}$: move $14_{6i},15_{6i}$ to stack $R$, then output all entries.

If one of $ 12_{6j},13_{6j}$ is missing, use the same procedure ignoring the missing entry.
		\qed\end{proof}

\begin{proposition}
	\label{prop:GiBases}
	 Let $G_i'$ be a permutation obtained by removing a single entry from $G_i$. Then $G'_i\in S(3,\infty)$.
\end{proposition}
\begin{proof}
We give a deterministic procedure to sort $G_i'$. There are three cases depending on from where the entry is removed. 

\smallskip\noindent
\textit{Term removed from $P$.}
		Let $P'$ be the factor $P$ with one entry removed.
We claim that there is a sorting sequence for $P' x_0$ 
which outputs the smallest six items in order and leaves $10,9$ in stack $R$.
To show this we simply consider all cases.
\begin{enumerate}
	\item If $1$ is removed,  $2,4,3$ can be output in order, then $7,6$ placed in stack $L$, $10$ in stack $R$, then $5,6,7$ output, and $9$ placed on top of $10$ in stack $R$.
	
	\item  If $2,3$, or $4$ are removed, write   $P'=ab761$ with  $a,b\in\{2,3,4\}$. Place $a,b$ in stack $R$, move $7,6$ into stack $L$, output $1$, then output $a,b$ in the correct order, then move $10$ into stack $R$, output $5,6,7$ and move $9$ into stack $R$.
		
	\item  If $6$ or $7$ is removed,   write  $P'=243a1$ with  $a\in\{7,6\}$. Place $4,3,2$ in stack $L$ in order, move $a$ into stack $R$, output 1 then $2,3,4$, then move $a$ into stack $L$, move $10$ into stack $R$, output $5,a$ and move $9$ into stack $R$.
\end{enumerate}
	
Thus after inputting $P'x_0$ we have the configuration shown in Fig.~\ref{fig:Premoved2} with $j=0$, which we can sort by Lemma~\ref{lem:can_sort1}.

\smallskip\noindent
\textit{Term removed from $x_s,  0\leq s\leq i$.}
	
	Input $P$ leaving $6,7$ in stack $R$, which brings us to the configuration in Fig.~\ref{fig:67} with $j=0$. 
 Now assume we have input $P\dots x_{j-1}y_{j-1}$ with  $j\leq s$ (note  the convention that $x_{-1},y_{-1}$ are empty) and the configuration is as in Fig.~\ref{fig:67}.
			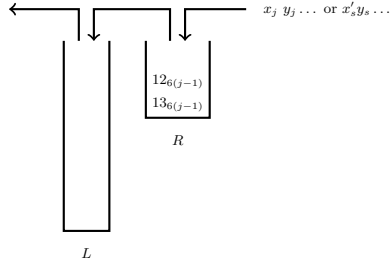
\begin{figure}[ht]
	\begin{center}
		\begin{tikzpicture}[thick,scale=0.6, every node/.style={scale=0.6}]
		\draw[thick] (0,1.7) -- (0,-2.5) -- (1,-2.5) -- (1,1.7);
		\draw (2.5,-.5) node {$R$};
		\draw[thick] (1.8,1.7) -- (1.8,0) -- (3.2,0) -- (3.2,1.7);
		\draw (0.5,-3) node {$L$};
		\draw[ ->] (4,2.4) -- node[above] {} (2.66,2.4) -- (2.66,1.7);

		\draw (5.8,2.4) node {$x_j\ y_j\dots \text{ or } x_s'y_s\dots $};
		\draw[bend left, ->] (2.33,1.7) -- (2.33,2.4) -- node[above] {} (0.66,2.4) -- (0.66,1.7);
		\draw[ <-] (-1.2,2.4) -- node[above] {} (0.33,2.4) -- (0.33,1.7);

		\draw (2.5,.8) node {$12_{6(j-1)}$};
		\draw (2.5,0.3) node {$13_{6(j-1)}$};
		\end{tikzpicture}
		\caption{Configuration after $P\dots x_{j-1}y_{j-1}$ is input }
		\label{fig:67}
	\end{center}
\end{figure}

	If $j<s$ we can input $x_jy_j$ into the stacks to arrive at the same configuration with $j$ incremented by $1$, as follows: move $10_{6j}$ to stack $L$, output $5_{6j},6_{6j}=12_{6(j-1)},7_{6j}=13_{6(j-1)}$, move $9_{6j}$ to stack $L$, move $13_{6j},12_{6j}$ to stack $R$, output $8_{6j},9_{6j},10_{6j}$.

	If $j=s$, we proceed as follows:
		\begin{enumerate}\item If $5_{6s}$ removed, output $6_{6s}=12_{6(s-1)},7_{6s}=12_{6(s-1)}$, move $9_{6s},10_{6s}$ to stack $R$, to reach the configuration in Fig.~\ref{fig:Premoved2} with $j=s$. From here the remaining entries can be sorted by Lemma~\ref{lem:can_sort1}.
		\item If $10_{6s}$ is removed, output $5_{6s}, 6_{6s}, 7_{6s}$ and place $9_{6s}$ in stack $R$,  to reach the configuration in Fig.~\ref{fig:Premoved2} with $j=s$ and $10_{6s}$ missing. From here the remaining entries can be sorted  Lemma~\ref{lem:can_sort1}.
		\item If $9_{6s}$ is removed, move $6_{6s}$ to stack $L$, move $10_{6s}$ on top of $7_{6s}$ in stack $R$, output $5_{6s}, 6_{6s}$, move $13_{6s}, 12_{6s}$ into $L$, then output $8_{6s}, 10_{6s}$. This gives the configuration in Fig.~\ref{fig:67_2} with $j=s$. From here the remaining entries can be sorted by Lemma~\ref{lem:can_sort2}.
	\end{enumerate}

\smallskip\noindent
\textit{Term removed from  $y_s, 0\leq s\leq i$ or $S_i$.}
Input $Px_0$ to reach the configuration in Fig.~\ref{fig:Si_removed} with $j=0$: move $2,3,4$ into stack $L$, $7,6$ to $R$, output $1,2,3,4$, move $10$ into $L$, output $5,6,7$ then move $9$ into $L$.

		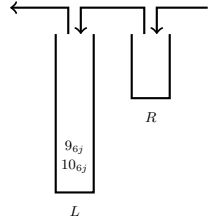
\begin{figure}[ht]
		\begin{center}
			\begin{tikzpicture}[thick,scale=0.5, every node/.style={scale=0.6}]
			\draw[thick] (0,1.7) -- (0,-2.5) -- (1,-2.5) -- (1,1.7);
			\draw (2.5,-.5) node {$R$};
			\draw[thick] (2,1.7) -- (2,0) -- (3,0) -- (3,1.7);
			\draw (0.5,-3) node {$L$};
			\draw[ ->] (4,2.4) -- node[above] {} (2.66,2.4) -- (2.66,1.7);
			\draw[bend left, ->] (2.33,1.7) -- (2.33,2.4) -- node[above] {} (0.66,2.4) -- (0.66,1.7);
			\draw[ <-] (-1.2,2.4) -- node[above] {} (0.33,2.4) -- (0.33,1.7);	
			\draw (0.5,-1.8) node {$10_{6j}$};
			\draw (0.5,-1.3) node {$9_{6j}$};

			\end{tikzpicture}

				\caption{Configuration after  $P x_0 y_0\ \dots \ x_{j}$ is input }
			\label{fig:Si_removed}
				\end{center}
	\end{figure}

Now suppose we have input $P x_0 y_0\ \dots \ x_{j}$ to reach the configuration in  Fig.~\ref{fig:Si_removed}.
If  no entry is removed from $y_j$ and $j<i$ then 
we can input $y_jx_{j+1}$ to return to the  configuration in Fig.~\ref{fig:Si_removed} with $j$ incremented by $1$ as follows: move $13_{6j},12_{6j}$ to stack $R$, output $8_{6j},9_{6j},10_{6j}$, move $10_{6(j+1)}$ to $L$, output $5_{6(j+1)}=11_{6j}, 12_{6j}, 13_{6j}$, then move $9_{6(j+1)}$ to stack $L$.

If $j=s$ ($y_s$ is removed):
\begin{enumerate}\item If $8_{6s}$ is removed, output $9_{6s},10_{6s}$, move $13_{6s},12_{6s}$ to stack $L$ to reach the configuration in Fig.~\ref{fig:67_2}, from which the remaining entries can be sorted by Lemma~\ref{lem:can_sort2}.
\item If $b\in\{13_{6s},12_{6s}\}$ is removed, place $b$ in stack $R$, output $8_{6s},9_{6s},10_{6s}$, move $b$ to stack $L$  to reach the configuration in Fig.~\ref{fig:67_2} with one of $12_{6s},13_{6s}$ removed, from which the remaining entries can be sorted a by Lemma~\ref{lem:can_sort2}.
\end{enumerate}

	If $j=i$ and the entry is removed from $S_i$, 			
	sort the remaining entries as follows:
	\begin{enumerate}
		\item If $11_{6i}$ is removed,  place $13_{6i},12_{6i}$ into stack $R$, output $8_{6i},9_{6i},10_{6i}$, then $12_{6i},13_{6i},14_{6i},15_{6i}$.
		
				\item If $b\in\{14_{6i}, 15_{6i}\}$ is removed,  place $13_{6i},12_{6i}$ into stack $R$, output $8_{6i},9_{6i},10_{6i}$, move  $12_{6i}$ into stack $L$, place $b$ on top of $13_{6i}$ in stack $R$, output $11_{6i}$ then $12_{6i}$, move $b$ into stack $L$, output $13_{6i}$ then $b$.
	\end{enumerate}	
\qed\end{proof}

\begin{theorem}
	\label{thm:GiBases}
	The set of permutations that can be sorted by a stack of depth 3 and an infinite stack in series has an infinite basis.
\end{theorem}
\begin{proof}
	Proposition~\ref{prop:Gi}  shows that each $G_i$ cannot be sorted, and Proposition~\ref{prop:GiBases} shows that no $G_i$ can contain $G_j$ for $j\neq i$ as a subpermutation since any subpermutation of $G_i$ can be sorted. 
	Thus $\mathscr G=\{G_{i}\mid i\in\N\}$ is 
	an infinite antichain in the basis for $S(3,\infty)$.
\qed\end{proof}

\section{From finite to infinitely based}

Let $\mathscr B_t$ be the basis for $S(t,\infty)$ for $t\in \N_+$.
Modifying Lemma 1 in \cite{Estacks} for the sorting case, we have the following:
\begin{lemma}\label{lem:basisk} If $\sigma\in\mathscr B_t$ has length $n$ then either  $\sigma$ or $(213)_n\sigma$
belongs to  $\mathscr B_{t+1}$. \end{lemma}
\begin{proof}
If $\sigma\not\in S(t+1,\infty)$ then since $\sigma\in \mathscr B_t$, deleting any entry gives a permutation  in $S(t,\infty)\subseteq S(t+1,\infty)$, so $\sigma\in\mathscr B_{t+1}$.
Else $\sigma\in S(t+1,\infty)$. In  any sorting process for $(213)_n\sigma$ the entries $1_n,2_n,3_n$ cannot appear together in stack $L$, so at least one entry must remain in stack $R$ which means we must sort $\sigma$ with stack $R$ of depth at most $t$, which is not possible, so 
$(213)_n\sigma$ cannot be sorted. If we remove an entry of the prefix then the two entries $a,b\in \{1_n,2_n,3_n\}$ can be placed in stack $L$ in order, leaving stack $R$ depth $t+1$ so the permutation can be sorted, and if an entry is removed from $\sigma$ then since $\sigma\in \mathscr B_t$ it can be sorted with $R$ having one space occupied.
\qed\end{proof}

\begin{theorem}
	\label{thm:LongestHeightSortingMachine}
	The set of permutations that can be sorted using a stack of depth $t\in\N_+$ and an infinite stack in series is finitely based if and only if $t\in\{1,2\}$.
\end{theorem}
\begin{proof} 
We have $|\mathscr B_1|=1$ and  $|\mathscr B_2|=20$  \cite{Knuth,Estacks}.
Theorem~\ref{thm:GiBases} shows that $\mathscr B_3$ is infinite. Lemma~\ref{lem:basisk} implies if $\mathscr B_t$ is infinite then so is $\mathscr B_{t+1}$.   
\qed\end{proof}

A small modification of 
Propositions~\ref{prop:Gi} and \ref{prop:GiBases} shows that   for $t\geq 4$ the set $\mathscr G_t=\{G_{i,t}\}$, where 
$G_{i,t}= P(x_0y_0)\dots (x_iy_i) (14 \ 15 \ 16 \ \dots \ 12_t \ 11)_{6i},$
is an  explicit antichain in the basis of $S(t,\infty)$. Details can be seen in \cite{GohThesis}.

\bibliographystyle{splncs03}
\bibliography{refs}

\end{document}